\documentclass[12pt, a4paper, twoside, reqno]{amsart}

\usepackage[margin=2.85cm]{geometry}

\usepackage[T1]{fontenc}
\usepackage[utf8]{inputenc}
\usepackage{amsmath}
\usepackage{amsthm}
\usepackage{amsfonts}
\usepackage{MnSymbol}
\usepackage{bbm}
\usepackage[UKenglish]{babel}
\usepackage{enumerate}
\usepackage[dvips]{graphicx}
\usepackage{bm}
\usepackage{ae}

\usepackage[draft,english]{fixme}

\theoremstyle{definition}

\theoremstyle{remark}

\theoremstyle{plain}
\newtheorem{thm}{Theorem}

\newtheorem{cor}[thm]{Corollary}

 \newcommand{\NN}{\mathbb{N}}
 \newcommand{\ZZ}{\mathbb{Z}}
 \newcommand{\QQ}{\mathbb{Q}}
 \newcommand{\RR}{\mathbb{R}}

 \newcommand{\abs}[1]{\left\lvert #1 \right\rvert}
 
 \newcommand{\inv}[1][1]{^{- #1}}
 \newcommand{\set}[1]{\left\lbrace #1 \right\rbrace}

\DeclareMathOperator{\supp}{supp}
\DeclareMathOperator{\dist}{dist}
\DeclareMathOperator{\Bad}{Bad}
\DeclareMathOperator{\DI}{DI}

\begin{document}

\title{Some remarks on Mahler's classification in higher dimension}

\author{S. KRISTENSEN, S. H. PEDERSEN, B. WEISS}

\address{S. KRISTENSEN, Department of Mathematics, Aarhus University,
  Ny Munkegade 118, DK-8000 Aarhus C, Denmark}

\email{sik@math.au.dk}

\address{S. H. PEDERSEN, Department of Mathematics, Aarhus University,
  Ny Munkegade 118, DK-8000 Aarhus C, Denmark}

\email{steffenh@math.au.dk}

\address{B. WEISS, 	Department of Mathematics, Tel Aviv University, Tel-Aviv, 69978 Israel}

\email{barakw@post.tau.ac.il}

\thanks{SK and SHP supported by the Danish Natual Science Research Council.}


\begin{abstract}
 We prove a number of results on the metric and non-metric theory of
 Diophantine approximation for Yu's multidimensional variant of
 Mahler's classification of transcendental numbers. Our results arise as applications of well known results in Diophantine approximation to the setting of Yu's classification.
 \end{abstract}

\maketitle

\section{Introduction}
\label{sec:introduction}

In \cite{Mahler32}, Mahler introduced a classification of
transcendental numbers in terms of their approximation properties by
algebraic numbers. More precisely, he introduced for each $k \in
\mathbb{N}$ and each $\alpha \in \mathbb{R}$ the Diophantine exponent 
\begin{multline}
\label{eq:1}
\omega_k(x) = \sup \{\omega \in \mathbb{R} \colon \vert P(x)\vert \le
H(P)^{-\omega} \\ \text{ for infinitely many irreducible } P \in \mathbb{Z}[X],
\deg(P)\le k \}. 
\end{multline}
Here, $H(P)$ denotes the naive height of the polynomial $P$, i.e. the
maximum absolute value among the coefficients of $P$. 

Mahler defined classes of numbers according to the asymptotic
behaviour of these exponents as $k$ increases. More precisely, let  
\begin{equation*}
\omega(x) = \limsup_{k\rightarrow \infty} \frac{\omega_k(x)}{k}.
\end{equation*}
The number $x$ belongs to one of the following four classes.
\begin{itemize}
\item $x$ is an $A$-number if $\omega(x) =0$, so that $x$ is algebraic
  over $\mathbb{Q}$. 
\item $x$ is an $S$-number if $0 < \omega(x) < \infty$.
\item $x$ is a $T$-number if $\omega(x) = \infty$, but $\omega_k(x) <
  \infty$ for all $k$. 
\item $x$ is a $U$-number if $\omega(x) = \infty$ and $\omega_k(x) =
  \infty$ for all $k$ large enough. 
\end{itemize}
All four classes are non-empty, with almost all real numbers being
$S$-numbers. Every real number belongs to one of the classes, and the
classes are invariant under algebraic operations over $\mathbb{Q}$.  

In analogy with Mahler's classification, Koksma \cite{Koksma39}
introduced a different classification based on the exponent 
\begin{equation*}
\omega^*_k(\alpha) = \sup \{\omega^* \in \mathbb{R} \colon \vert x-
\alpha \vert \le H(\alpha)^{-\omega^*} \text{ for infinitly many }
\alpha \in \overline{\mathbb{Q}}\cap \mathbb{R}, \deg(\alpha)\le k
\}. 
\end{equation*}
In this case, $H(\alpha)$ denotes the naive height of $\alpha$,
i.e. the naive height of the minimal integer polynomial of $\alpha$. In analogy with Mahler's classification, one
defines $w^*(x)$ and $A^*$-, $S^*$-, $T^*$- and $U^*$-numbers. 

The reader is referred to the monograph \cite{bugeaud} for an
excellent overview of the classifications and their properties. A
particular property is that the classifications coincide, so that
$A$-numbers are $A^*$-numbers, $S$-numbers are $S^*$-numbers and so
on. The individual exponents however need not coincide. 

In \cite{Yu87}, Yu introduced a  classification similar to Mahler's
for $d$-tuples of real numbers. In brief, the classification is
completely similar, except that the exponents $\omega_k(x)$ are now
defined in terms of integer polynomials in $d$ variables.

An analogue of Koksma's classification was introduced by Schmidt
\cite{Sch06}. However, the relation between the two classifications is
not at all clear, and it is conjectured that the two classifications
do not agree \cite{Sch06}. 

It is the purpose of the present note to study the Diophantine
approximation problems arising within Yu's classification. We recall
the simple connection between the questions arising from Mahler's
classification, and the problem of Diophantine approximation with
dependent quantities. A classical problem in Diophantine
approximation, given $\mathbf{x} = (x_1, \ldots, x_d) \in
\mathbb{R}^d$, is to find $\omega$ for which 
\begin{equation}\label{eq: new label}
\Vert \mathbf{q} \cdot \mathbf{x} \Vert
% = \Vert q_1 x_1 + \cdots + q_n
%x_n \Vert 
\le 
%H(\mathbf{q})
(\max_{1 \leq i \leq d} |q_i|)^{-\omega} \text{ for infinitely many }
\mathbf{q} = (q_1, \ldots, q_d)\in \mathbb{Z}^d, 
\end{equation}
where as usual $\Vert \cdot \Vert$ denotes the distance to the nearest
integer. Comparing (\ref{eq:1}) and (\ref{eq: new label}), one sees
that one can define Mahler's exponents $\omega_k$ by restricting the
classical problem to a consideration of vectors 
%To obtain the exponents $\omega_n$ in Mahler's
%classification, one should restrict the vector 
$\mathbf{x}$ belonging to the Veronese curve  
\begin{equation*}
\Gamma = \left\{ (x, x^2, \dots, x^k) \in \mathbb{R}^k : x \in
  \mathbb{R}\right\}. 
\end{equation*}
Similarly, in order to understand the exponents arising in Yu's
classification, one should once more consider the 
corresponding problem of a single linear form, but replace the
Veronese curve by the variety obtained by letting the coordinates
consist of the distinct non-constant monomials in $d$ variables of total degree at most
$k$, say. The resulting Diophantine approximation properties
considered in this case would correspond to the multidimensional
analogue of $\omega_k$, i.e. 
\begin{multline*}
\omega_k(\mathbf{x}) = \sup \{\omega \in \mathbb{R} \colon \vert
P(\mathbf{x})\vert \le H(P)^{-\omega} \text{ for infinitely many }\\ 
 P \in \mathbb{Z}[X_1, \dots, X_d], \deg(P)\le k \}.
\end{multline*}

Throughout, let $n = \binom{k+d}{d}-1$ be the number of nonconstant
monomials in $d$ variables of total degree at most $k$. In addition to
the usual, naive height $H(P)$, we will also use the following
modification $\tilde{H}(P)$, which is the maximum absolute value of
the coefficients of the non-contant terms of $P$. The following is a
slight re-statement of \cite[Theorem 1]{Yu87}. 

\begin{thm}
\label{thm:Dir}
For any $\mathbf{x} = (x_1, \dots, x_d) \in \mathbb{R}^d$, there
exists $c(k, \mathbf{x})> 0$ such that for all $Q > 1$, there is a
polynomial $P \in \mathbb{Z}[X_1, \dots, X_d]$ of total degree at most
$k$ and height $H(P) \le Q$, such that  
\begin{equation*}
\vert P(\mathbf{x}) \vert < c(k, \mathbf{x}) Q^{-n}.
\end{equation*}
Replacing the condition $H(P) \le Q$ by $\tilde{H}(P) \le Q$, we may
always choose $c(k, \mathbf{x}) = 1$. 
\end{thm}

The proof is essentially an application of the pigeon hole principle,
and is completely analogous to the classical proof of Dirichlet's
approximation theorem in higher dimension. As a standard corollary,
one obtains the first bounds on the exponents $\omega_k(\mathbf{x})$. 

\begin{cor}
\label{cor:Dir}
For any $\mathbf{x} = (x_1, \dots, x_d) \in \mathbb{R}^d$, there
exists a $c(k, \mathbf{x}) > 0$ such that 
\begin{equation*}
\vert P(\mathbf{x}) \vert < c(k, \mathbf{x}) H(P)^{-n},
\end{equation*}
for infinitely many $P \in \mathbb{Z}[X_1, \dots, X_d]$ of total
degree at most $k$. In particular, $\omega_k(\mathbf{x}) \ge n$. 
\end{cor}
The corollary tells us what the normalising factor in the
multidimensional definition of $\omega(\mathbf{x})$ should be, namely the
number of non-constant monomials in $d$ variables of total degree at
most $k$. 

Inspired by the above result, we will define the notions of $k$-very
well approximable, $k$-badly approximable, $k$-singular and
$k$-Dirichlet improvable. We will then proceed to prove that the set
defined in this manner are all Lebesgue null-sets and so are indeed
exceptional. In the case of $k$-badly approximable results, we will
also show that these form a thick set, i.e. a set whose intersection
with any ball has maximal Hausdorff dimension. In fact, many of our
results are somewhat stronger than these statements. The properties
are all consequences of other work by various authors (see
below). Finally, we will deduce a Roth type theorem from Schmidt's
Subspace Theorem \cite{Sch80}. 

It is not the aim of the present paper to prove deep results concerning Yu's classification, but rather to examine the extent to which already existing methods have something interesting to say about the classification.

\section{Results and proofs}
In each of the following subsections we introduce a property of
approximation of $d$-tuples of real numbers by algebraic numbers, and
prove a result about it which extends previous results known in case
$d=1$.

\subsection{$k$-very well approximable points}
	A point $\mathbf{x} = (x_1, \dots, x_d) \in \RR^d$ is called
        \emph{$k$-very well approximable} if there exists $\varepsilon > 0$
        and infinitely many polynomials $P \in \ZZ[X_1, \dots, X_d]$
        of total degree at most $k$, such that  
	\begin{equation}\label{eq: new again}
		\abs{P(\mathbf{x})} \leq H(P)^{-(n+\varepsilon)}.	
	\end{equation}
In other words, $\mathbf{x}$ is $k$-very well approximable if the
exponent $n$ on the right hand side in Corollary \ref{cor:Dir} can be
increased by a positive amount. We will prove that this property is
exceptional in the sense that almost no points with respect to the
$d$-dimensional Lebesgue measure are $k$-very well approximable. In
fact, we will show that this property is stable under restriction to
subsets supporting a measure with nice properties. 

We recall some properties of measures from \cite{KLW04}. A measure
$\mu$ on $\RR^d$ is said to be \emph{Federer} (or doubling) if there
is a number $D > 0$ such that for any $x \in \supp(\mu)$ and any $r >
0$, the ball $B(x,r)$ centered at $x$ of radius $r$ satisfies
\begin{equation}
\label{eq:federer}
\mu\big(B(x,2r)\big) < D \mu\big(B(x,r)\big).
\end{equation}
The measure $\mu$ is said to be \emph{absolutely decaying} if for some
pair of numbers $C, \alpha > 0$ 
\begin{equation}
\label{eq:abs_decay}
\mu\left(B(x,r) \cap \mathcal{L}^{(\varepsilon)}\right) \le C
\left(\frac{\varepsilon}{r}\right)^\alpha \mu\big(B(x,r)\big), 
\end{equation}
for any ball $B(x,r)$ with $x \in \supp(\mu)$ and any affine
hyperplane $\mathcal{L}$, where $\mathcal{L}^{(\varepsilon)}$ denotes
the $\varepsilon$-neighbourhood of $\mathcal{L}$.  
A weaker variant of the property of being absolutely decaying is
obtained by replacing $r$ in the denominator on the right hand side of
\eqref{eq:abs_decay} by the quantity 
\begin{equation*}
\sup\{c > 0 \colon \mu(\{z \in B(x,r) : \dist(z,\mathcal{L})>c\}) > 0\}.
\end{equation*}
In this case, we say that $\mu$ is \emph{decaying}.
If the measure $\mu$ has the property that 
\begin{equation}
\label{eq:nonplanar}
\mu(\mathcal{L})=0,
\end{equation}
for any affine hyperplane $\mathcal{L}$, $\mu$ is called
\emph{non-planar}. Note that an absolutely decaying measure is
automatically non-planar, but a decaying measure need not be
non-planar. Finally, $\mu$ is called \emph{absolutely friendly}
if it is Federer and absolutely decaying, and is called \emph{friendly}
if it is Federer, decaying, and non-planar. 

	\begin{thm}
	Let $\mu$ be an absolutely decaying Federer measure on
        $\RR^d$. For any $k \in \NN$, the set of $k$-very well
        approximable points is a null set with respect to $\mu$. In
        particular, Lebesgue almost-no points are $k$-very well
        approximable.			 
	\end{thm}
Our proof relies on results of \cite{KLW04}, in which the case $d=1$
was proved. 	
	\begin{proof}
	Let $f : \RR^d \to \RR^n$ be defined by $f(x_1, \dots, x_d) = (x_1,x_2, \dots,
        x_{d-1}x_d^{k-1} ,x_d^k)$, so that $f$ maps $(x_1, \dots, x_d)$
        to the $n$ distinct nonconstant monomials in $d$ variables of
        total degree at most $k$. Clearly, $f$ is smooth, and by
        taking partial derivatives, we easily see that $\RR^n$ may be
        spanned by the partial derivatives of $f$ of order up to $k$. 

 From \cite[Theorem 2.1(b)]{KLW04} we immediately see that the
 pushforward $f_*\mu$ is a friendly measure on $\mathbb{R}^n$. We now apply
 \cite[Theorem 1.1]{KLW04}, which states that a friendly measure is
 strongly extremal, i.e. for any $\delta > 0$, almost no points in the
 support of the measure have the property that  
 \begin{equation*}
 \prod_{i=1}^n \vert q y_i - p_i\vert < q^{-(1+\delta)},
 \end{equation*}
 for infinitely many $\mathbf{p} \in \ZZ^n$, $q \in \NN$. Clearly,
 this implies the weaker property of extremality, i.e. that for any
 $\delta' > 0$,  almost no
 points in the support of the measure satisfy
  \begin{equation}
  \label{eq:simul}
 \max_{1 \le i \le n} \vert q y_i - p_i\vert < q^{-(\frac{1}{n}+\delta')},
 \end{equation}
for infinitely many $\mathbf{p} \in \ZZ^n$, $q \in \NN$. 

To get from the above to a proof of the theorem, we need to
re-interpret this in terms of polynomials. We apply Khintchine's
transference principle \cite[Theorem V.IV]{cassels} to see that
\eqref{eq:simul} is satisfied infinitely often if and only if  
\begin{equation}
\label{eq:form}
\vert \mathbf{q} \cdot \mathbf{y} - p \vert < H(\mathbf{q})^{-(n+\delta'')},
\end{equation}
for infinitely many $\mathbf{q} \in \ZZ^n$, $p \in \ZZ$, where
$\delta'' >0$ can be explicitly bounded in terms of $n$ and
$\delta'$. Now, $\mathbf{y}$ lies in the image of $f$, so that the
coordinates of $\mathbf{y}$ consist of all monomials in the variables $(x_1,
\dots x_d)$, whence any polynomial in these $d$ variables may be
expressed on the form $P(\mathbf{x}) = \mathbf{q} \cdot \mathbf{y} -
p$. 
The coefficients of $P$ include all the coordinates of $\mathbf{q}$
and hence $H(P) \geq H(\mathbf{q})$, so that if (\ref{eq: new again})
holds for infinitely many 
$P$ with $\varepsilon = \delta''$, then \eqref{eq:form} holds for
infinitely many $\mathbf{q}, p$. 
Since the latter condition is satisfied on a set of $\mu$-measure
zero, it follows that $\mu$-almost all points in $\mathbb{R}^d$ are
not $k$-very well approximable. 

The final statement of the theorem follows immediately, as the
Lebesgue measure clearly is Federer and absolutely decaying. 
	\end{proof}

Some interesting open questions present themselves at this stage. One can ask whether a vector exists which is $k$-very well approximable for all $k$. We will call such vectors \emph{$k$-very very well approximable}. It is not difficult to prove that the set of $k$-very well approximable vectors is a dense $G_\delta$-set, so the question of existence can be easily answered in the affirmative. However, determining the Hausdorff dimension of the set of very very well approximable vectors is an open question. When $d = 1$, it is known that the Hausdorff dimension is equal to $1$ due to work of Durand \cite{Durand08}, but the methods of that paper do not easily extend to larger values of $d$.

Taking the notion one step further, one can ask whether vectors $\mathbf{x} \in \mathbb{R}^d$ exist such that for some fixed $\varepsilon > 0$, for any $k \in \mathbb{N}$, there are infinitely many integer polynomials $P$ in $d$ variables of total degree at most $k$, such that
\begin{equation*}
	\abs{P(\mathbf{x})} \leq H(P)^{-(n+\varepsilon)},
\end{equation*}
where as usual $n = \binom{n+d}{d} - 1$, i.e. in addition to $\mathbf{x}$ being very very well approximable, we require the very very very significant improvement in the rate of approximation to be uniform in $k$. We will call such vectors \emph{very very very well approximable}. Determining the Hausdorff dimension of the set of very very very well approximable numbers is an open problem.

\subsection{$k$-badly approximable points}
	A point $\mathbf{x} = (x_1, \dots, x_d) \in \RR^d$ is called \emph{$k$-badly approximable} if there exists $C = C(k, \mathbf{x})$ such that
	\[
		\abs{P(\mathbf{x})} \geq C H(P)^{-n},
	\]
	for all non-zero polynomials $P \in \ZZ[X_1, \dots, X_d]$ of
        total degree at most $k$. In other words, a point $\mathbf{x}
        \in \RR^d$ is $k$-badly approximable if the approximation rate
        in Corollary \ref{cor:Dir} can be improved by at most a
        positive constant in the denominator.  Let $B_k$ be the set of
        $k$-badly approximable points. Note that each set $B_k$ is a null set, which is easily
        deduced from the work of Beresnevich, Bernik, Kleinbock and
        Margulis \cite{BBKM}. 
	We will now show: 

	\begin{thm}
		Let $B \subseteq \RR^d$ be an open ball and let $M \in \NN$. Then
		\[
			\dim B \cap \bigcap_{k=1}^M B_k = d.
		\]
	\end{thm}
	This statement is deduced from the work of Beresnevich
        \cite{Ber13}, who proved the case $d=1$. 
	
	\begin{proof}
		Let $n_k = \binom{k+d}{d}-1$ as before, but with the
                dependence on $k$ made explicit in notation. Let $f :
                \RR^d \to \RR^{n_M}$ be given by $f(x_1, \dots, x_d) =
                (x_1,x_2, \dots, x_{d-1}x_d^{M-1} ,x_d^M)$, with the
                monomials ordered in blocks of increasing total
                degree.   Let $\mathbf{r}_k = (\frac{1}{n_k}, \dots,
                \frac{1}{n_k}, 0, \dots , 0) \in \RR^{n_M}$, where the
                non-zero coordinates are the first $n_k$ coordinates,
                so that $\mathbf{r}_k$ is a probability vector. 
		
We define as in \cite{Ber13} the set of $\mathbf{r}$-approximable
points for a probability vector $\mathbf{r}$ to be the set 
\begin{multline*}
\Bad(\mathbf{r}) = \Big\{\mathbf{y} = (y_1, \dots, y_{n_M}) : \text{
  for some } C(\mathbf{y}) > 0,\\ 
 \max_{1 \le i \le n_M} \Vert q y_i \Vert^{1/r_i} \ge C(\mathbf{y}) q^{-1},
 \text{ for any } q \in \NN\Big\}. 
\end{multline*}
Here, $\Vert z \Vert$ denotes the distance to the nearest integer, and
we use the convention that $z^{1/0} = 0$. 
		
Let $1 \le k \le M$ be fixed	 and let $\mathbf{x} \in \RR^d$
satisfy that $f(\mathbf{x}) \in \Bad(\mathbf{r}_k)$. From \cite[Lemma
1]{Ber13}) it follows, that there exists a constant $C = C(k,
\mathbf{x})$, such that the only integer solution $(a_0, a_1, \dots,
a_{n_k})$ to the system 
		\[
			\abs{a_0 + a_1 x_1 + a_2 x_2 + \dots 
				+ a_{n_k-1} x_{d-1}x_d^{k-1} + a_{n_k} x_d^k} < C H^{-1},
			\quad \max_i \abs{a_i} < H^{1/n_k}
		\]
		is zero. Here, the choice of $\mathbf{r}_k$ and the
                ordering of the monomials in the function $f$ ensure
                that the effect of belonging to $\Bad(\mathbf{r}_k)$
                will only give a polynomial expression of total degree
                at most $k$. Indeed, writing out the full equivalence,
                we would have the first inequality unchanged, with the
                second being $\max_i \abs{a_i} < H^{r_{k, i}}$, where the
                exponent is the $i$'th coordinate of
                $\mathbf{r}_k$. If this coordinate is $0$, we are only
                considering polynomials where the corresponding $a_i$
                is equal to zero. 
		
		Rewriting this in terms of polynomials, for any
                non-zero $P \in \ZZ[X_1, \dots, X_d]$ with $H(P) <
                H^{1/n_k}$ and total degree at most $k$,  we must have  
		\[
			\abs{P(\mathbf{x})} \geq C H^{-1} > C H(P)^{-n_k}.
		\]
		
		It follows that $\mathbf{x} \in B_k$, and hence
                $f\inv(\Bad(\mathbf{r}_k)) \subseteq B_k$. The result
                now follows by applying \cite[Theorem 1]{Ber13}, which
                implies that the Hausdorff dimension of the
                intersection of the sets $f\inv(\Bad(\mathbf{r}_k))$
                is maximal.  
	\end{proof}

Again, an interesting open problem presents itself, namely the question of uniformity of the constant $C(k, \mathbf{x})$ in $k$. Is it possible to construct a vector in $B_k$ for all $k$ with the constant being the same for all $k$? And in the affirmative case, what is the Hausdorff dimension of this set? A weaker version of this question would be to ask whether there is some natural dependence of $C(k, \mathbf{x})$ on $k$, i.e. whether one can choose $C(k, \mathbf{x}) = C(\mathbf{x})^k$ or a similar dependence. We do not at present know the answer to these questions.

\subsection{$(k,\varepsilon)$-Dirichlet improvable vectors and $k$-singular vectors}

Let $\varepsilon > 0$. A point $\mathbf{x}$ is called \emph{$(k,
\varepsilon)$-Dirichlet improvable} if for any $\varepsilon$ there
exists a $Q_0 \in \NN$, such that for any $Q \geq Q_0$ there exists a
polynomial $P \in \ZZ[X_1, \dots, X_d]$ with total degree at most $k$,
\[
		\tilde{H}(P) \leq \varepsilon Q \ \text{ and } \  
	\abs{P(\mathbf{x})} \leq \varepsilon Q^{-n}. 
	\]
Note that we are now using $\tilde{H}$ as a measure of the complexity
of our polynomials. 
 	
In view of Theorem \ref{thm:Dir}, if $\varepsilon \ge 1$, all points
clearly have this property, and so the property is only of interest
when $\varepsilon < 1$.  A vector is called \emph{$k$-singular} if it is $(k,
\varepsilon)$-Dirichlet improvable for every $\varepsilon>0$. 

We will need a few additional definitions before proceeding. For a
function $f : \RR^d \rightarrow \RR^n$, a measure $\mu$ on $\RR^d$ and
a subset $B \in \RR^d$ with $\mu(B) > 0$, we define  
\[
\Vert f \Vert_{\mu,B} = \sup_{\mathbf{x} \in B \cap \supp \mu} \vert
f(\mathbf{x}) \vert. 
\]
Let $C, \alpha > 0$ and let $U \subseteq \RR^d$ be open. We will say
that the function $f$ is \emph{$(C, \alpha)$-good with respect to $\mu$ on
$U$} if for any ball  $B \subseteq U$ with centre in $\supp \mu$ and
any $\varepsilon > 0$, 
\[
\mu\left\{\mathbf{x} \in B : \vert f(\mathbf{x}) \vert < \varepsilon
\right\} \le C \left(\frac{\varepsilon}{\Vert f
    \Vert_{\mu,B}}\right)^\alpha \mu\big(B\big). 
\]

We will say that a measure $\mu$ on $\RR^d$ is \emph{$k$-friendly} if it is
Federer, non-planar and the function $f : \RR^d \to \RR^n$  given by
$f(x_1, \dots, x_d) = (x_1,x_2, \dots, x_{d-1}x_d^{k-1} ,x_d^k)$ is
$(C, \alpha)$-good with respect to $\mu$ on $\RR^d$ for some $C,\alpha
> 0$. 

We have 
	\begin{thm}\label{thm: new}
		Let $\mu$ be a $k$-friendly measure on $\RR^d$. Then
                there is an $\varepsilon_0 = \varepsilon_0(d,\mu)$
                such that the set of $(k, \varepsilon)$-Dirichlet
                improvable points has measure zero for any
                $\varepsilon < \varepsilon_0$. In particular, the set
                of $k$-singular vector has measure zero. 
	\end{thm}

In the case when $d=1$, $k \ge 2$ and $\mu$ being the Lebesgue measure on $\mathbb{R}$, the result is immediate from work of Bugeaud \cite[Theorem 7]{Bugeaud02}, in which an explicit value of $\varepsilon$ is given, namely $\varepsilon = 2^{-3k-3}$. Our proof is non-effective and relies on  \cite[Theorem 1.5]{KW08}.
	
	\begin{proof}
		Under the assumption on the measure $\mu$, 
                \cite[Theorem 1.5]{KW08} implies the existence of an
                $\varepsilon_0 > 0$ such that for all
                $\tilde{\varepsilon} < \varepsilon_0$ 
		\[
			f_* \mu(\DI_{\tilde{\varepsilon}}(\mathcal{T}))=0	
			\text{ for any unbounded }\mathcal{T} \subseteq \mathfrak{a}^+	.
		\]
		
	Here, $f$ is the usual function $f(x_1, \dots, x_d) =
        (x_1,x_2, \dots, x_{d-1}x_d^{k-1} ,x_d^k)$, $\mathfrak{a}^+$
        denotes the set of $(n+1)$-tuples of $(t_0, t_1, \dots, t_n)$
        such that $t_0 = \sum_{i=1}^n t_i$, $t_i>0$ for each $i$, and
        $\DI_{\tilde{\varepsilon}}(\mathcal{T})$ denotes the set of
        vectors $\mathbf{y} = (y_1, \dots, y_n) \in \RR^n$ for which
        there is a $T_0$ such that for any $t \in \mathcal{T}$ with
        $\Vert t \Vert \ge T_0$, the system of inequalities 
	\begin{equation*}
	\begin{cases}
	\vert \mathbf{q}\cdot \mathbf{y} - p \vert< \tilde{\varepsilon} e^{-t_0} &\\
	\vert q_i \vert < \tilde{\varepsilon} e^{t_i} & i=1, \dots, n,
	\end{cases}
	\end{equation*}
has infinitely many non-trivial integer solutions $(\mathbf{q}, p) =
(q_1, \dots , q_n, p) \in \ZZ^{n+1}\setminus \{0\}$. 
	
		Our result follows by specialising the above
                property. Indeed, we apply this to $\varepsilon =
                \tilde{\varepsilon}^{n+1} < \varepsilon_0^{n+1}$ and
                the central ray in $\mathfrak{a}^+$,  
		\[
		\mathcal{T} = \set{\left(t,\frac{t}{n}, \dots,
                    \frac{t}{n} \right) : t =
                  \log\left(\frac{Q}{\tilde{\varepsilon}} \right)n, Q
                  \geq [\varepsilon_0]+1, Q \in \NN}. 
		\]
		The measure $f_* \mu$ is the pushforward under $f$ of
                the $k$-friendly measure $\mu$. It follows that the
                set of $\mathbf{x} \in \RR^d$ for which their image
                under $f$ is in
                $\DI_{\tilde{\varepsilon}}(\mathcal{T})$ is of measure
                zero for all $\tilde{\varepsilon} <
                \varepsilon_0^{n+1}$. 
		From the definition of $\DI_{\tilde{\varepsilon}}$ and
                the choice of $\mathfrak{a}^+$ and $\mathcal{T}$,
                $f(\mathbf{x}) \in \DI_{\tilde{\varepsilon}}$ if and
                only if there is a $Q_0 \ge \max\{[\varepsilon_0]+1,
                \tilde{\varepsilon}e^{T_0/n}\}$, such that for $Q >
                Q_0$  there exists $q_0, q_1, \dots, q_n \in \ZZ$ with
                $ \max_{1\leq i\leq n} \abs{q_i} < \tilde{\varepsilon}
                e^{t/n} = Q$, such that   
		\[
		\abs{(q_1, \dots q_n)\cdot f(\mathbf{x}) + q_0} <
                \tilde{\varepsilon} e^{-t}=\varepsilon Q^{-n}. 
		\]
Reinterpreting the right hand side of the above as a polynomial
expression in $\mathbf{x}$, this recovers the exact definition of
$\mathbf{x}$ being $(k,\varepsilon^{1/(n+1)})$-Dirichlet improvable. 
	\end{proof}

Note that the proof in fact yields a stronger statement. Namely, by
adjusting the choice of $\mathfrak{a}^+$, we could have put different
weights on the coefficients of the approximating polynomials, thus
obtaining the same result, but with a non-standard (weighted) height
of the polynomial. 

As with the preceding results, some open problems occur. We do not at present know if there exist a vector $\mathbf{x}$, for which there are positive numbers $\varepsilon_k > 0$, such that $\mathbf{x} \in \DI(k, \varepsilon_k)$. If this is the case, determining the Hausdorff dimension of the set of such vectors is another open problen. Additionally, the same questions can be asked if we require $\varepsilon$ to be independent of $k$, i.e. if we ask for the existence of a vector $\mathbf{x} \in \DI(k, \varepsilon)$ for all $k$.

Let us now say that $\mathbf{x} \in \mathbb{R}^d$ is \emph{$k$-algebraic} if there exists
a nontrivial polynomial $P \in \mathbb{Z}[X_1, \ldots, X_d]$ of degree at most
$k$, such that $P(\mathbf{x})=0$. It is clear that if $\mathbf{x}$ is
$k$-algebraic, then it is $k$-singular.  
In light of Theorem \ref{thm: new}, it is natural to inquire whether
all $k$-singular points are $k$-algebraic. In this direction we
have:
\begin{thm}\label{thm: another}
For $d \geq 2$, for any $k \geq 1$, there exists a  $k$-singular point
in $\mathbb{R}^d$ which is not $k$-algebraic. 
\end{thm}
The proof relies on results of \cite{KW_singular}. For $d=1$, much less appears to be known in general. For $k=2$, it follows from a result of Roy \cite{roy_square} combined with a transference result (see \cite[Theorem V.XII]{cassels}) that the answer is affirmative. Roy further indicates in \cite{roy_cube} that he has an unpublished result for $k=3$, which would imply the analogue of Theorem \ref{thm: another} in the case $d=1$, $k=3$. Already for $k=2$, the construction is rather involved and a general approach would be desirable.

%We do not know
%whether the conclusion of Theorem \ref{thm: another} is valid for
%$d=1$. 

\begin{proof}
Once more, for a fixed $k$, we take $f$ as in the proof of Theorem \ref{thm: new}. 
In the notation of \cite{KW_singular}, it is clear that $\mathbf{x}
\in \mathbb{R}^d$
is $k$-singular if $f(\mathbf{x}) \in
\mathrm{Sing}(\mathbf{n})$. Also $f(\mathbf{x})$ is totally irrational
in the notation of \cite{KW_singular} if and only if $\mathbf{x}$ is
$k$-algebraic. 

 Since the image of $f$ is
a $d$-dimensional nondegenerate analytic submanifold of
$\mathbb{R}^n$, for $d \geq 2$ we can apply \cite[Theorem
1.2]{KW_singular} to conclude that the intersection of $f \left(\mathbb{R}^d
\right)$ with $\mathrm{Sing}(\mathbf{n})$ contains a totally
irrational point.  
\end{proof}

Theorem \ref{thm: new}  does not give an explicit value of $\varepsilon_0$,
and indeed the value depends on the measure $\mu$. However we can at
least push $\varepsilon_0$ 
to the limit $\varepsilon_0 \nearrow 1$ in the case when $\mu$ is
the Lebesgue measure on $\RR^d$ 
to obtain a result on the $k$-singular vectors. 

	\begin{thm}
	For any $d$, the set of $\mathbf{x}$ which are $(k,
        \varepsilon)$-Dirichlet improvable for some $\varepsilon <1$
        and some $k$, has Lebesgue measure zero.
	\end{thm}
The proof relies on the work of Shah \cite{Shah}. 

\begin{proof}
This is a direct consequence of \cite[Corollary 1.4]{Shah}, where the
set $\mathcal{N}$ is chosen to be the diagonal $\mathcal{N} = \{(N,
\dots, N) : N \in \mathbb{N}\}$. 
\end{proof}
	
	Note that once again, the result of Shah gives a stronger
        result in the sense that we may take a non-standard height as
        in the preceding case and retain the conclusion.

\subsection{Algebraic vectors}

Our final result, which is again a corollary of known results, is an
analogue of Roth's Theorem \cite{Roth}, which states that algebraic
numbers are not very well approximable. Schmidt's Subspace Theorem,
see e.g. \cite{Sch80}, provides a higher dimensional analogue of this
result, and it is this theorem we will apply. We will say that a
vector $\boldsymbol{\alpha} = (\alpha_1, \dots, \alpha_d) \in \RR^d$
is \emph{algebraic of total degree $k$} if there is a polynomial
$P_{\boldsymbol{\alpha}} \in \ZZ[X_1, \dots, X_d]$ of total degree $k$
with $P_{\boldsymbol{\alpha}}(\boldsymbol{\alpha}) = 0$ and if no
polynomial of lower total degree vanishes at $\boldsymbol{\alpha}$. 

	\begin{thm}
		Let $\boldsymbol{\alpha} = (\alpha_1, \dots, \alpha_d)
                \in \RR^d$ be an algebraic $d$-vector of total degree
                more than $k$. Then for any $\varepsilon > 0$ there
                are only finitely many non-zero polynomials $P \in
                \ZZ[X_1, \dots, X_d]$ of total degree at most $k$ with 
		\[
			\abs{P(\boldsymbol{\alpha})} < H(P)^{-(n+\varepsilon)}.	
		\]
	\end{thm}
	
	\begin{proof}
		Since $\boldsymbol{\alpha}$ in not algebraic of total
                degree at most $k$, by definition it follows that the
                numbers $1, \alpha_1, \alpha_2, \dots,
                \alpha_{d-1}\alpha_d^{k-1} ,\alpha_d^k$ are
                algebraically independent over $\QQ$. From a corollary
                to Schmidt's Subspace Theorem, \cite[Chapter VI
                Corollary 1E]{Sch80}, it follows that there are only
                finitely many non-zero integer solutions $(q_0, \dots,
                q_n)$ to  
		\[
			\abs{q_0 + q_1 \alpha_1 + q_2 \alpha_2 + \dots+ 
			q_{n-1} \alpha_{d-1}\alpha_d^{k-1} + q_n \alpha_d^k} 
			< (\max_{1 \leq i \leq n}{\abs{q_i}})^{-(n+\varepsilon)}.
		\] 
This immediately implies the result.
	
	\end{proof}

\nocite{*}
\bibliographystyle{amsplain}
\bibliography{paper_modified_barak}

\providecommand{\bysame}{\leavevmode\hbox to3em{\hrulefill}\thinspace}
\providecommand{\MR}{\relax\ifhmode\unskip\space\fi MR }
\providecommand{\MRhref}[2]{%
  \href{http://www.ams.org/mathscinet-getitem?mr=#1}{#2}
}
\providecommand{\href}[2]{#2}
\begin{thebibliography}{10}

\bibitem{Ber13}
V.~Beresnevich, \emph{Badly approximable points on manifolds}, Invent. Math.
  \textbf{202} (2015), no.~3, 1199--1240. \MR{3425389}

\bibitem{BBKM}
V.~V. Beresnevich, V.~I. Bernik, D.~Y. Kleinbock, and G.~A. Margulis,
  \emph{Metric {D}iophantine approximation: the {K}hintchine-{G}roshev theorem
  for nondegenerate manifolds}, Mosc. Math. J. \textbf{2} (2002), no.~2,
  203--225, Dedicated to Yuri I. Manin on the occasion of his 65th birthday.
  \MR{1944505 (2004b:11107)}

\bibitem{Bugeaud02}
Y.~Bugeaud, \emph{Approximation by algebraic integers and {H}ausdorff
  dimension}, J. London Math. Soc. (2) \textbf{65} (2002), no.~3, 547--559.
  \MR{1895732 (2003d:11110)}

\bibitem{bugeaud}
Y.~Bugeaud, \emph{Approximation by algebraic numbers}, Cambridge University
  Press, 2004.

\bibitem{cassels}
J.~W.~S. Cassels, \emph{An introduction to diophantine approximation},
  Cambridge University Press, 1957.

\bibitem{Durand08}
A.~Durand, \emph{Large intersection properties in {D}iophantine approximation
  and dynamical systems}, J. Lond. Math. Soc. (2) \textbf{79} (2009), no.~2,
  377--398. \MR{2496520 (2010b:11084)}

\bibitem{KLW04}
D.~Y. Kleinbock, E.~Lindenstrauss, and B.~Weiss, \emph{On fractal measures and
  {D}iophantine approximation}, Selecta Math. (N.S.) \textbf{10} (2004), no.~4,
  479--523. \MR{2134453 (2006g:11151)}

\bibitem{KW_singular}
D.~Y. Kleinbock and B.~Weiss, \emph{Friendly measures, homogeneous flows and
  singular vectors},  \textbf{385} (2005), 281--292. \MR{2180240 (2006f:11084)}

\bibitem{KW08}
\bysame, \emph{Dirichlet's theorem on {D}iophantine approximation and
  homogeneous flows}, J. Mod. Dyn. \textbf{2} (2008), no.~1, 43--62.
  \MR{2366229 (2008k:11078)}

\bibitem{Koksma39}
J.~F. Koksma, \emph{{\"U}ber die mahlersche klasseneinteilung der
  transzendenten zahlen und die approximation komplexer zahlen durch
  algebraische zahlen}, Monatsh. Math. Phys. \textbf{48} (1939), 176--189.

\bibitem{Mahler32}
K.~Mahler, \emph{Zur approximation der exponentialfunktionen und des
  logarithmus, i, ii}, J. Reine Angew. Math. \textbf{166} (1932), 118--150.

\bibitem{Roth}
K.~F. Roth, \emph{Rational approximations to algebraic numbers}, Mathematika
  \textbf{2} (1955), 1--20; corrigendum, 168. \MR{0072182 (17,242d)}

\bibitem{roy_square}
D.~Roy, \emph{Approximation simultan\'ee d'un nombre et de son carr\'e}, C. R.
  Math. Acad. Sci. Paris \textbf{336} (2003), no.~1, 1--6.

\bibitem{roy_cube}
\bysame, \emph{On simultaneous rational approximations to a real number, its
  square, and its cube}, Acta Arith. \textbf{133} (2008), no.~2, 185--197.

\bibitem{Sch80}
W.~M. Schmidt, \emph{Diophantine approximation}, Lecture Notes in Mathematics,
  vol. 785, Springer, Berlin, 1980. \MR{568710 (81j:10038)}

\bibitem{Sch06}
\bysame, \emph{Mahler and {K}oksma classification of points in {$\Bbb R^n$} and
  {$\Bbb C^n$}}, Funct. Approx. Comment. Math. \textbf{35} (2006), 307--319.
  \MR{2271620 (2008c:11105)}

\bibitem{Shah}
N.~A. Shah, \emph{Expanding translates of curves and {D}irichlet-{M}inkowski
  theorem on linear forms}, J. Amer. Math. Soc. \textbf{23} (2010), no.~2,
  563--589. \MR{2601043 (2011g:11137)}

\bibitem{Yu87}
K.~R. Yu, \emph{A generalization of {M}ahler's classification to several
  variables}, J. Reine Angew. Math. \textbf{377} (1987), 113--126. \MR{887404
  (88h:11049)}

\end{thebibliography}

\end{document}